\documentclass{amsart}
\usepackage{amssymb}
\usepackage{amsfonts}

\newtheorem{theorem}{Theorem}
\theoremstyle{plain}

\newtheorem{definition}{Definition}

\newtheorem{proposition}{Proposition}
\newtheorem{remark}{Remark}

\numberwithin{equation}{section}
\input{tcilatex}

\newcommand{\term}[1]{\textit{#1}}
\newcommand{\mh}{\textsc{m}}
\newcommand{\ah}{\textsc{a}}
\newcommand{\lh}{\textsc{l}}
\newcommand{\eh}{\textsc{e}}

\newcommand{\dlh}{\bar{\textsc{l}}}

\newcommand{\R}{{\mathbb{R}}}
\newcommand{\N}{{\mathbb{N}}}

\begin{document}
\title[Homogeneity]{Homogeneity\\
}
\author{Martin Himmel}
\curraddr{Technical University Dresden,
 “Friedrich List” Faculty of Transport and Traffic Sciences,
  Institute of Transport and Economics, 
 Würzburger Str. 35,
 01187 Dresden,
 Germany}
\email{martin.himmel@tu-dresden.de}

\begin{abstract}
The four types of homogeneity - additive, multiplicative, exponential, and logarithmic -
are generalized as transformations describing how a function changes under scaling or shifting of its arguments. 
These generalized homogeneity functions capture different scaling behaviors and establish fundamental properties of $f$.

Such properties include how homogeneity is preserved under function operations and how it determines the transformation behavior of related constructions like Cauchy quotients.
This framework extends the classical concept of homogeneity to a wider class of functional symmetries, providing a unified approach to analyzing scaling properties in various mathematical contexts.
\end{abstract}

\maketitle

\QTP{Body Math}
$\bigskip $

\section{Introduction}
Let $T \subset (0,+\infty)$ be a set which is closed with respect to multiplication and $I \subset \R$ an interval.
A function $f: I \to \R$ is called homogeneous, 
if there is a function $m: T \to \R$ such that
\begin{equation*}
f(tx)=m(t)f(x)
\end{equation*}
holds for all $x\in I$ and $t \in T$.
We take this definition as a starting point to investigate a more general notion, namely

\begin{definition}
Let $I \subset \R$ be a non-empty interval and $T \subset \R$ a semigroup with respect to multiplication such that $T I \subset I $.
A function $f: I \to \R$ is called homogeneous with respect to the function $\mh: I \times T \to \R$, if
\begin{equation*}
f(tx)=\mh(x,t) f(x), \qquad x \in I, t \in T.
\end{equation*}
In this case we call $\mh$ \term{homogeneity function} of $f$. %
If $T=I$, since multiplication of real numbers is commutative, 
we have
\begin{equation*}
\textsc{m}(x,t) f(x) = \textsc{m}(t,x) f(t), \qquad x \in I, t \in T,
\end{equation*}
and thus
\begin{equation*}
\textsc{m}(x,t)  = \frac{f(t)}{f(x)} \textsc{m}(t,x)  , \qquad x \in I, t \in T.
\end{equation*}
\end{definition}

Obviously, the homogeneity function $\textsc{m}$ of a  function $f: I \to \R$ is given by
\begin{equation*}
\textsc{m}(x,t)=\frac{f(tx)}{f(x)}, \qquad x \in I, t \in T,
\end{equation*} 
whenever $f(x) \neq 0$; moreover, the homogeneity function is undetermined at the zeros of $f$.
%
Similarly, neglecting for a moment domain issues, a function $f$ is called translative, if
there is a function $a$ such that
\begin{equation*}
f(x+t)=a(t)+f(x),
\end{equation*}
for all $x$ and $t$,
motivating us to introduce translativity functions in the following

\begin{definition}
Let $I \subset \R$ be a non-empty interval and $T \subset \R$ a semigroup with respect to addition such that $I+T  \subset I$.
A function $f: I \to \R$ is called translative, or \textit{additively homogeneous}, 
with respect to the function $\textsc{a}: I \times T  \to \R$, if
\begin{equation*}
f(x+t)=\textsc{a}(x,t)+ f(x), \qquad x \in I, t \in T.
\end{equation*}
In this case we call $\textsc{a} =\textsc{a}_f$ translativity function (or, more systematically, an additive homogeneity function) of $f$.
\end{definition}

Two other notions, which seem very natural from the perspective of functional equations,
are exponential and logarithmic homogeneity.

\begin{definition}
Let $I \subset \R$ be a non-empty interval and $T \subset \R$ be a semigroup with respect to addition.
A function $f: I \to \R$ is called \textit{exponentially homogeneous} with respect to the function $\eh: I \times T \to \R$, if
\begin{equation*}
f(x+t)=\eh(x,t) f(x), \qquad x \in I, t \in T.
\end{equation*}
In this case we call $\eh$ (or $\eh_f$ to emphasize the dependency on the function $f$) an exponential homogeneity function of $f$.
\end{definition}

\begin{definition}
Let $I \subset \R$ be a non-empty interval and $T \subset \R$ be a semigroup with respect to multiplication.
A function $f: I \to \R$ is called \textit{logarithmically homogeneous} with respect to the function $\textsc{l}: I \times T \to \R$, if
\begin{equation*}
f(tx)=\textsc{l}(x,t) + f(x), \qquad x \in I, t \in T.
\end{equation*}
In this case we call $\textsc{l}$ a logarithmic homogeneity function of $f$.
\end{definition}

\begin{remark}
When clear from the context, we sometimes suppress 
the dependency on $f$ in the homogeneity function; otherwise we write
\begin{align*}
\textsc{a}_{f} (x,t)=f(x+t) - f(x)
\end{align*}	
for the \textit{additive homogeneity function},

\begin{align*}
	\textsc{m}_{f} (x,t)=\frac{f(tx)}{f(x)}  
\end{align*}	
for the \textit{multiplicative homogeneity function},
\begin{align*}
	\eh_{f} (x,t)=\frac{f(x+t)}{f(x)}  
\end{align*}	
for the \textit{exponential homogeneity function}, and

\begin{align*}
	\textsc{l}_{f} (x,t)=f(tx)-f(x)  	
\end{align*}	
for the \textit{logarithmic homogeneity function}.
Note that for multiplicative and exponential homogeneity functions some care has to be taken due to the fact that $f$ appears in the denominator.
Moreover, for multiplicative and logarithmic homogeneity functions we find it more natural to write $tx$ instead of $xt$, which of course does not matter, as long we are in domains of real numbers, where multiplication is commutative. 
Without mentioning explicitly we assume
everything to be well-defined, 
which implies suitable relations between the intervals $I$ and $T$ to hold (cf. Definition 1 to 4).
The input function $f$ of the respective homogeneity functions is also referred to as \term{generator}.
\end{remark}

\subsection{Dual Homogeneity Function}
It is natural to introduce also
four other types of homogeneity functions,
each of them complementing the just introduced notion of homogeneity.
In a nutshell, these four more notions of homogeneity functions 
come to pass when subtraction (division) is replaced by addition (multiplication).
These dual notions of homogeneity not only make sense 
from a formal point of view, but also are of practical use
when investigating algebraic properties of homogeneity
in Section \ref{AlgProp}.

\begin{remark}
Under reasonable domain assumptions, we call
\begin{align*}
\bar{\ah}_{f} (x,t)=f(x+t) + f(x)
\end{align*}	
 the \textit{additive dual homogeneity function},

\begin{align*}
	\bar{\mh}_{f} (x,t)={f(tx)}{f(x)}  
\end{align*}	
 the \textit{multiplicative dual homogeneity function},
\begin{align*}
	\bar{\eh}_{f} (x,t)={f(x+t)}{f(x)}  
\end{align*}	
 the \textit{exponential dual homogeneity function}, and

\begin{align*}
	\bar{\lh}_{f} (x,t)=f(tx) + f(x)  	
\end{align*}	
the \textit{logarithmic dual homogeneity function}.

\end{remark}

\subsection{Relation between Types of Homogeneity}
Under suitable assumptions on the function domain these four notions of homogeneity are equivalent.
Let $I = (0,+\infty)$ and $f: I \to (0,+\infty)$ a function and $\textsc{m}:I \times T \to (0,+\infty)$,
$\textsc{m}(x,t)=\frac{f(tx)}{f(x)}$ its (multiplicative) homogeneity function.
Since $I \subset (0,+\infty)$, for every $x \in I$ there is a $u \in \R$ with $e^u=x$.
Similarly,  if $T \subset (0,+\infty)$,  for every $t \in T$ there is $s \in \R$ with $e^s=t$. 
Thus,  
\begin{eqnarray*}
\textsc{m}_f (x,t) &= \textsc{m}_{ f}(e^u,  e^s)\\
&=
\frac{f(e^{u+s})}{f(e^u)}\\
&=
\eh_{f \circ \exp }(u,s).
\end{eqnarray*}
Hence,  $f$ is multiplicatively homogeneous of degree $\textsc{m}$ iff $g:=f \circ \exp :\R \to (0 ,+ \infty)$
is exponentially homogeneous of degree $\eh_g :\R^2 \to (0, +\infty) $ defined by 
$\eh_g (u,s):=\frac{g(u+s)}{g(s)}$. 
In other words,  
the multiplicative homogeneity function of $f$ is the exponential homogeneity function of $g$.
So for positive functions when dealing with positive homogeneity
we may equivalently deal with exponential homogeneity on the whole real numbers .\\

Similarly,  the function $f: I \to \R$ is translative (or, in other words, additively homogeneous) of degree
$\textsc{a}: I \times T \to \R$,  defined by $\textsc{a}(x,t)=f(x+t)-f(x)$,  
with $I, T \subset \R$, if and only if,  
since for every $x \in I$ and for every $t \in T$ there exist $u,s \in (0,+\infty)$ with $x=\log{u}$ and $t=\log{s}$,
the function $g:=f \circ \log : (0,+\infty) \to \R$ is lagrithmically homogeneous of degree
$\textsc{l}:  (0,+\infty)^2 \to \R $ defined by $\textsc{l}(u,s):=\textsc{a}_f (\log{u},  \log{s} )$,
or equivalently, 
iff the function $h:=\exp \circ f \circ \log: (0,+\infty) \to (0,+\infty) $ is multiplicatively homogeneous of degree 
 $\textsc{m}_h :  (0,+\infty)^2 \to (0,+\infty)$.

\subsection{Symmetry of Homogeneity Functions}
In general, the form of homogeneity functions 
excludes symmetry.
Homogeneity functions are symmetric only for constant functions.
For multiplicative and logarithmic  homogeneity functions, respectively, 
the result relies on the commutativity of multiplication;
for additive and exponential homogeneity functions, respectively,
it is based on
commutativity of addition.

\begin{remark}
Assume $\textsc{m}_f(x,t)= \textsc{m}_f(t,x)$ for all $x$ and $t$,
where $f: I \to (0, + \infty)$ and $I \subset \R$ an interval.
By the definition of multiplicative homogeneity functions,
\begin{equation*}
\frac{f(tx)}{f(x)}=
\frac{f(xt)}{f(t)}, \qquad x,t.
\end{equation*}
Since real multiplication is abelian,
we have $f(x)=f(t)$ for all $x$ and $t$.
For $x=t$, this is no restricting condition.
Otherwise, namely if $f(x)=f(t)$ for all $x \neq t$%
\footnote{Recall that in this paper we assume $T$ and $I$ to be real intervals. },
describes the function to be "maximally non-injective",
or, in other words, $f$ is constant.
\end{remark}

The proof for logarithmic homogeneity functions
is very similar.
\begin{remark}
Assume $\textsc{l}_f(x,t)= \textsc{l}_f(t,x)$ for all $x$ and $t$.
By the definition of logarithmic homogeneity functions,
\begin{equation*}
{f(tx)}-{f(x)}=
{f(xt)}-{f(t)}.
\end{equation*}
we have $f(x)=f(t)$ for all $x$ and $t$,
implying $f(x)=f(t)$ for all $x \neq t$.
Thus, $f$ is constant.
\end{remark}

The proof for the other two types of homogeneity functions
relies heavily on the commutativity of addition.
\begin{remark}
Assume $\textsc{a}_f(x,t)= \textsc{a}_f(t,x)$ for all $x$ and $t$.
Thus,
\begin{equation*}
{f(x+t)}-{f(x)}=
{f(t+x)}-{f(t)}, \qquad x,t.
\end{equation*}
\begin{center}
$\Leftrightarrow$
\end{center}
\begin{equation*}
{f(x)}={f(t)}, \qquad x,t.
\end{equation*}
\begin{center}
$\Leftrightarrow$
\end{center}

\begin{center}
$f$ is constant.
\end{center}
\end{remark}
The proof for exponential homogeneity functions is left to the reader.

\subsection{Algebraic Properties of Homogeneity Functionals}
\label{AlgProp}
Here we state some algebraic properties such as
linearity or multiplicativity
of homogeneity functions when considered as functionals.

\subsubsection{Additive Homogeneity}
\begin{remark}[Linearity of $\ah_{\cdot}(x,t)$]
The additive homogeneity functional is linear, i.e., it is
\begin{enumerate}
    \item 
{additive}:

\begin{align*}
\ah_{f+g}(x,t) &=
({f+g})(x+t)-{(f+g)(x)}\\
&=
f(x+t)+g(x+t)-f(x)-g(x) \\
&=
f(x+t)-f(x) + g(x+t)-g(x)\\
&=
\ah_{f}(x,t) + \ah_{g}(x,t), \qquad x,t;
\end{align*}

   \item 
homogeneous:
\begin{align*}
\ah_{\lambda f}(x,t) &=
({\lambda f})(x+t)-{(\lambda f)(x)}\\
&=
\lambda f(x+t)-\lambda f(x) \\
&=
\lambda \left(f(x+t)- f(x) \right) \\
&=
\lambda \cdot \ah_{f}(x,t), \qquad x,t;
\end{align*}
\end{enumerate}
in general,
$a_{\cdot}(x,t)$ is not multiplicative:
\begin{align*}
\ah_{fg}(x,t) &=
({fg})(x+t)-{(fg)(x)}\\
&=
f(x+t) g(x+t)-f(x) g(x) \\
&\neq 
\ah_{f}(x,t) \cdot \ah_{g}(x,t)\\
&=
\left( {f(x+t)-f(x)} \right) \left( {g(x+t)-g(x)} \right), \qquad x,t;
\end{align*}
\end{remark}

Since the additive homogeneity functional is not multiplicative in general,
it is reasonable to 
\begin{itemize}
    \item 
ask when (i.e. for which functions $f$ and $g$) it is  multiplicative,
 \item 
 investigate $\ah_{fg}(x,t) - \ah_{f}(x,t) \cdot \ah_{g}(x,t)$
 (or  
 $\frac{\ah_{fg}(x,t)}{\ah_{f}(x,t) \cdot \ah_{g}(x,t)}$)
 measuring the deviation of the additive homogeneity functional from being multiplicative,

\end{itemize}

 For given $f,g$, it may also be fruitful to address 
 whether the homogeneity functional is multiplicative on a smaller set
 (\term{conditional multiplicativity}).
 

A somewhat more suited notion for the additive homogeneity functional is the concept of dual multiplicativity.

\begin{remark}[Dual Multiplicativity of $\ah_{\cdot}(x,t)$]
The additive homogeneity functional is dual multiplicative
if, and only if, the functions are proportional, i.e.,
if there is $c \in \R$ such that $f= c g$.
\end{remark}

\begin{proof}
\begin{align*}
\ah_{fg}(x,t) &=
({fg})(x+t)-{(fg)(x)}\\
&=
f(x+t) g(x+t)-f(x)g(x) \\
&=
\underbrace{ (f(x+t) - f(x)) }_{ = \ah_f(x,t) } 
\underbrace{ (g(x+t) + g(x))}_{ = \bar{\ah}_g(x,t)} - \underbrace{( f(x+t) g(x) - f(x) g(x+t)}_{=:R_{f,g} (x,t)}, \qquad x,t;
\end{align*} 
If $R_{f,g} (x,t)$ vanishes, i.e., if $f(x+t) g(x) = f(x) g(x+t)$, 
or equivalently,
\begin{align*}
\frac{f(x+t)}{g(x+t)}= \frac{f(x)}{g(x)}, \qquad x,t,
\end{align*}
the additive homogeneity functional is dual multiplicative.
Setting $t:=-x+y$ and $h:=\frac{f}{g}$, we get that the dual multiplicativity of the additive homogeneity functional is equivalent to 
$h$ being constant, which means $f=c g$ for some $c \in \R$ as claimed.
The converse is easy to verify.
\end{proof}

\subsubsection{Exponential Homogeneity}
\begin{remark}[Multiplicativity of $\eh_{\cdot}(x,t)$]
The exponential homogeneity functional is multiplicative, and thus homogeneous, but not additive.
\end{remark}

\begin{proof}
\begin{align*}
\eh_{fg}(x,t) &=
\frac{({fg})(x+t)}{{(fg)(x)}}\\
&=
\frac{ f(x+t) g(x+t)}{{f(x) g(x)}}\\
&=
\eh_{f}(x,t) \cdot \eh_{g}(x,t), \qquad x,t;
\end{align*}
\end{proof}

Since in general $\eh_{\cdot}(x,t)$ is not additive,
it is reasonable to take a closer look at
\begin{align*}
\eh_{f+g}(x,t) -\eh_{f}(x,t) - \eh_{g}(x,t) 
&=
- \frac{f(x+t) g(x)}{(f+g)(x) f(x)}
- \frac{g(x+t) f(x)}{(f+g)(x) g(x)}, \qquad x,t;
\end{align*}

\begin{remark}[Additivity of $\eh_{\cdot}(x,t)$]
The exponential homogeneity functional is additive only if $f \equiv 0$ or $f=-g$.
\end{remark}

\begin{proof}
Assuming $\eh_{f+g}(x,t) = \eh_{f}(x,t)+ \eh_{g}(x,t)$, setting $t=0$, we have
\begin{align*}
\left(\frac{f(x)}{g(x)}\right)^2 = -\left(\frac{f(x)}{g(x)}\right), \qquad x,t;
\end{align*}
Thus, the function $h:=\frac{f}{g}$
satisfies
\begin{align*}
h(x) \left( h(x)+1 \right)=0, \qquad x,t,
\end{align*}
which gives the claim.
\end{proof}

\subsubsection{Multiplicative Homogeneity}
\begin{remark}[Linearity of $\mh_{\cdot}(x,t)$]
Let $I,T \subset \R$ be intervals such that $T\cdot I \subset I$ and $1 \in T$.
The multiplicative homogeneity functional $\mh_{\cdot}(x,t)$  is
\begin{enumerate}
    \item 
additive only for zero functions or additive inverses of each other,
i.e., $f\equiv 0$ or $f=-g$;

   \item 
homogeneous of order $0$;

 \item
 never logarithmic;
  \item 
exponential if 
\[
g(x) = \frac{a f(x)}{f(x) - a}.
\]
for some non-zero $a \in \R$ such that $f(x) \neq a$ for all $x$ in the domain.
\end{enumerate}
\end{remark}

\begin{proof}
    
\begin{enumerate}
    \item 
Assuming additivity of $\mh_{\cdot}(x,t)$, we have, for all $x,t$,
\begin{align*}
\mh_{f+g}(x,t) &=
\frac{({f+g})(tx)}{{(f+g)(x)}}\\
&=
\frac{{f(tx)+g(tx)}}{{f(x)+g(x)}}\\
&=
\mh_{f}(x,t) + \mh_{g}(x,t)\\
&{=}
\frac{{f(tx)}}{{f(x)}} +\frac{{g(tx)}}{{g(x)}};
\end{align*}
Thus,
\begin{align*}
\frac{{f(tx)}}{{g(tx)}}
&=
-\left(\frac{{f(x)}}{{g(x)}}\right)^2, \qquad x,t.
\end{align*}
Setting $t=1$ we get
\begin{align*}
\left(\frac{f(x)}{g(x)}\right)^2 = -\left(\frac{f(x)}{g(x)}\right), \qquad x;
\end{align*}

Put $h:=\frac{{f}}{{g}}$; thus $h(x)=-(h(x))^2$,
or equivalently, $h(x) (h(x)+1)=0$ for all $x$.

By the definition of $h$, this means that $f \equiv 0$ (together with $g \neq 0$ arbitrary), or $f=-g$ as claimed.
   \item 
Obviously, for all $x,t$ and $\lambda \neq 0$,
\begin{align*}
\mh_{\lambda f}(x,t)
&=
\frac{\lambda{f(tx)}}{\lambda {f(x)}} \\
&=\mh_{f}(x,t),
\end{align*}
which means that
$\mh_{\cdot }(x,t)$ is homogeneous of order $0$,
thus homogeneous of order $1$ only for $\lambda=1$.

To see the latter, assume that $\mh$ is homogeneous of degree $p$ and $q$ with $p \neq q$: 
\begin{align*}
\mh_{ \lambda f}&=\lambda^p f\mh_{f}, \\
\mh_{ \lambda f}&=\lambda^q \mh_{f}.
\end{align*}
for $\lambda \in T$.
Thus,
\begin{align*}
\lambda^p \mh_{f}=\lambda^q \mh_{f}.
\end{align*}
Since $\mh_{f} \neq 0$, we obtain
$\lambda^{p-q} =1$,
which means that 
$\lambda$ is a $(p-q)$-th root of unity.

For $p=0$
we have $\lambda^{-q} =1$.
If $\lambda$ is assumed to be real, it follows
$\lambda =1$ as claimed.

 \item
 Assuming $\mh_{fg}(x,t) = \mh_{f}(x,t)+ \mh_{g}(x,t)$
 for all $x, t$,
 gives us
 \begin{align*}
\frac{f(tx) g(tx)}{f(x) g(x)}
&=
\frac{g(tx)}{g(x)}+ \frac{f(tx)}{f(x)}\\
 &=
 \frac{ g(tx)f(x) +  f(tx)g(x)}{g(x) f(x)}.
\end{align*}
Dividing both sides by the left-hand side yields
 \begin{align*}
 \frac{ g(tx)f(x) +  f(tx)g(x)}{f(tx) g(tx)} 
 &=
 \frac{f(x)}{f(tx)} +
 \frac{g(x)}{g(tx)}\\
&=
1,
\qquad x,t;
\end{align*}
Setting here $t=1$ gives $1+1=1$, a contradiction.
  \item 
For the multiplicative homogeneity functional to be exponential,
assume \( f \) and \( g \) such that
\begin{equation} \label{eq:main}
\eh_{f+g}(x, t) = \eh_f(x, t) \cdot \eh_g(x, t).
\end{equation}
Thus,
\[
\frac{f(tx) + g(tx)}{f(x) + g(x)} = \frac{f(tx)}{f(x)} \cdot \frac{g(tx)}{g(x)}.
\]
Multiplying both sides by \( f(x) + g(x) \) and \( f(x) g(x) \), we get
\[
[f(tx) + g(tx)] f(x) g(x) = f(tx) g(tx) [f(x) + g(x)].
\]
Rearranging gives us
\[
f(x) g(x) [f(tx) + g(tx)] = f(tx) g(tx) [f(x) + g(x)].
\]
Dividing by \( f(x) g(x) f(tx) g(tx) \) (assuming nonzero values):
\[
\frac{1}{f(tx)} + \frac{1}{g(tx)} = \frac{1}{f(x)} + \frac{1}{g(x)}.
\]
Thus, \( x \mapsto \frac{1}{f(x)} + \frac{1}{g(x)} \) is $0$-homogeneous 
(scale-invariant), hence constant:
\[
\frac{1}{f(x)} + \frac{1}{g(x)} = c.
\]
Solving for \( g(x) \) gives us
\[
g(x) = \frac{1}{c - \frac{1}{f(x)}} = \frac{f(x)}{c f(x) - 1}.
\]
under the assumption that $f(x) \neq \frac{1}{c}$ for all $x \in I$.
The case $c=0$ gives us $g(x)=-f(x)$ (cf. part (i) of this remark).
Without loss of generality, assume now $c \neq 0$.

Let \( c = \frac{1}{a} \neq 0 \) (and thus $a \neq 0$) to obtain
\[
g(x) = \frac{a f(x)}{f(x) - a}.
\]
with $\quad f(x) \neq a$ for all $x \in I$.

Conversely, if \( g(x) = \frac{a f(x)}{f(x) - a} \), then
\[
f(x) + g(x) = \frac{(f(x))^2}{f(x) - a}, \quad
f(tx) + g(tx) = \frac{f(tx)^2}{f(tx) - a},
\]
so
\[
\eh_{f+g}(x, t) = \frac{f(tx)^2 (f(x) - a)}{f(x)^2 (f(tx) - a)}.
\]
Also,
\[
\eh_f(x, t) = \frac{f(tx)}{f(x)}, \quad
\eh_g(x, t) = \frac{g(tx)}{g(x)} = \frac{f(tx) (f(x) - a)}{f(x) (f(tx) - a)},
\]
and their product is
\[
\eh_f(x, t) \cdot \eh_g(x, t) = \frac{f(tx)^2 (f(x) - a)}{f(x)^2 (f(tx) - a)} = \eh_{f+g}(x, t).
\]
This completes the proof.
\end{enumerate}
\end{proof}

\subsubsection{Logarithmic Homogeneity}
\begin{remark}[Linearity of $\lh_{\cdot}(x,t)$]
The logarithmic homogeneity functional $\lh_{\cdot}(x,t)$  is linear, i.e., it is
\begin{enumerate}
    \item 
additive and
 \item 
homogeneous;
\end{enumerate}
\item 
moreover, it is dual multiplicative only if the functions are proportional;  
\end{remark}

\begin{proof}
\begin{enumerate}
    \item 
To prove additivity of $\lh_{\cdot}(x,t)$, observe that
\begin{align*}
\lh_{f+g}(x,t) &= (f+g)(tx)-(f+g)(x)\\
&=
f(tx)+g(tx)-f(x)-g(x)\\
&=f(tx)-f(x)+g(tx)-g(x)\\
&= \lh_{f}(x,t) + \lh_{g}(x,t), \quad x,t.
\end{align*}
 \item 
For the homogeneity of $\lh_{\cdot}(x,t)$, verify that
\begin{align*}
\lh_{\lambda f}(x,t) &= (\lambda f)(tx)-(\lambda f)(x)\\
&=
 \lambda ( f(tx) - f(x)\\
 &=
\lambda \lh_{f}(x,t), \quad x,t.
\end{align*}

\item 
Assuming $\lh_{\cdot}(x,t)$ to be dual multiplicative, we get
\begin{align*}
\lh_{fg}(x,t) &= (fg)(tx)-(fg)(x)\\
&=
 f(tx) g(tx)-f(x) g(x)\\
&=
\underbrace{ (f(tx) - f(x)) }_{ = \lh_f(x,t) } 
\underbrace{ (g(tx) + g(x))}_{ = \dlh_g (x,t)} - \underbrace{( f(tx) g(x) - f(x) g(tx)}_{=:R_{f,g} (x,t)}, \qquad x,t;
\end{align*} 

Thus $R_{f,g} \equiv 0$, hence  $f(tx) g(x) = f(x) g(tx)$, 
or equivalently,
\begin{align*}
\frac{f(tx)}{g(tx)}= \frac{f(x)}{g(x)}, \qquad x,t.
\end{align*}
Setting $t:=\frac{y}{x}$ and $h:=\frac{f}{g}$, we get that the dual multiplicativity of the logarithmic homogeneity functional is equivalent to 
$h$ being constant, which means that there is $c \in \R$ such that $f=c g$.
The converse implication is easy to verify.
\end{enumerate}

\end{proof}
 
\subsubsection{Inversion of Generators}
In terms of inversion the following properties hold true:
\begin{align*}
\eh_{\frac{1}{f}} &=
\frac{1}{\eh_{f}},
\end{align*}
and 
\begin{align*}
\mh_{\frac{1}{f}} &=
\frac{1}{\mh_{f}};
\end{align*}

Analogously, we have for additive inverses of the generator 
for all $x,t$:
\begin{align*}
\ah_{-{f}} &=
-\ah_{{f}},
\end{align*}
and \begin{align*}
\lh_{-{f}}&=
-\lh_{{f}};
\end{align*}

\section{Some examples of homogeneity functions}
In this section we consider homogeneity functions corresponding to some classical generators.
We start with regular solutions to one of the Cauchy functional equations \cite[128-130]{Kuczma2}, namely:
\begin{itemize}
\item 
 additive functions: lines through the origin;
\item 
multiplicative functions: power functions; 
\item 
exponential functions: functions of exponential type;
\item 
 logarithmic functions: scaler multiples of logarithmic functions
\end{itemize}

\subsection{Additive function}
The (multiplicative) homogeneity function of $f:\R \to \R$ defined by $f(x)=cx$, for $c\in\R$, $c\neq 0$, arbitrarily fixed,
is clearly
\begin{align*}
\textsc{m}(x,t)&=\frac{f(tx)}{f(x)}\\
&=\frac{c({tx})}{cx}\\
&=t\\
&= \frac{f(t)}{c}\\
&= \frac{f(t)}{f(1)}, 
\end{align*}
the identity function on $\R \setminus \{0\}$ only depending on $t$.

The additive homogeneity function reads
\begin{align*}
\textsc{a}(x,t)&=f(x+t)-f(x)\\
&={c({x+t})}-{cx}\\
&=ct\\
&= f(t).
\end{align*}

The logarithmic homogeneity function reads
\begin{align*}
\textsc{l}(x,t)&=f(tx)-f(x)\\
&={c({tx})}-{cx}\\
&=cx (t-1)\\
&= f(x) \frac{f(t-1)}{f(1)}\\
&= f(x) \textsc{m} (x,t-1).
\end{align*}

The exponential homogeneity function reads
\begin{align*}
\eh(x,t)&=\frac{f(x+t)}{f(x)}\\
&=\frac{c(x+t)}{cx}\\
&=1 + \frac{t}{x}\\
&= 1 +s \\
&= 1 +\frac{f(t)}{f(x)}
\end{align*}
where $s:=\frac{t}{x}$ as above.

\subsection{Power functions}
The homogeneity function of $f:(0, +\infty) \to (0, +\infty)$ defined by $f(x)=x^p$, for $p\in \R$ arbitrarily fixed,
is clearly
\begin{align*}
\textsc{m}(x,t)&=\frac{f(tx)}{f(x)}\\
&=\frac{(tx)^p}{x^p}=t^p.
\end{align*}
Thus, the homogeneity function of a smooth multiplicative function is a multiplicative function independent of $x$
and of the same degree as $f$.

Its additive homogeneity function reads
\begin{align*}
\textsc{a}(x,t)&=f(x+t)-f(x)\\
&={(x+t)^p}-{x^p}.\\
\end{align*}

The logarithmic homogeneity function reads
\begin{align*}
\textsc{l}(x,t)&=f(tx)-f(x)\\
&={(tx)^p}-{x^p}\\
&=x^p (t^p -1).
\end{align*}

The exponential homogeneity function reads
\begin{align*}
\eh(x,t)&=\frac{f(x+t)}{f(x)}\\
&=\frac{(x+t)^p}{x^p}\\
&=\left(1 + \frac{t}{x}\right)^p.
\end{align*}

Note that, for power functions, $\eh$ is homogeneous of degree zero,
i.e. $\eh(sx,st)=\eh(x,t)$ for all $x,t,s \in (0, +\infty)$. 

Since the identity function is both additive and multiplicative,
 we have the following
\begin{remark}
Let $f: I \to \R $ with $f(x)=x$.
Then:
\begin{align*}
\ah_{id}(x,t) &= {f(x+t)}-{f(x)} \\
&= t, \qquad x,t;\\
\lh_{id}(x,t) &= \frac{f(tx)}{f(x)}\\
&= t, \qquad x,t;\\
\lh_{id}(x,t) &= {f(tx)}-{f(x)}\\
&= x(t-1), \qquad x,t;\\
\eh_{id}(x,t) &= \frac{f(x+t)}{f(x)}  \\
&=1 + \frac{t}{x}, \qquad  x \neq 0, t; 
\end{align*}
\end{remark}

\subsection{Exponential functions}
Let $f: \R \to (0, +\infty)$ with $f(x)=a^x$ for some $a>0$ be a smooth exponential function.
Its multiplicative homogeneity function reads
\begin{align*}
	\textsc{m}(x,t)&=\frac{f(tx)}{f(x)}\\
	&=\frac{a^{tx}}{a^{x}}\\
	&= a^{tx-x}\\
	&= a^{x(t-1)}\\
	&=
	f(x(t-1)).
\end{align*}

The additive homogeneity function reads
\begin{align*}
	\textsc{a}(x,t)&=f(x+t)-f(x)\\
	&=a^{x+t}-{a^{x}}\\
	&= a^{x} \left(a^t-1 \right)\\
	&= f(x) \left( f(t)-1 \right).
\end{align*}

The logarithmic homogeneity function reads
\begin{align*}
	\textsc{l}(x,t)&=f(tx)-f(x)\\
	&={a^{tx}}-{a^{x}}\\
	&=a^x \left(a^{tx-x}-1 \right)\\
	&=f(x) \left(f(x(t-1))-1 \right)\\
	&= f(x) \left(\textsc{m}_f {(x,t)}-1 \right).
\end{align*}

The exponential homogeneity function reads
\begin{align*}
	\eh(x,t)&=\frac{f(x+t)}{f(x)}\\
	&=\frac{a^{x+t}}{a^x}=a^t\\
	&=	f(t).
\end{align*}

\subsection{Logarithmic functions}
The homogeneity function of $f:(0, +\infty) \to \R$ defined by $f(x)=c \log{x}$, for arbitrary $c\in \R$, $c \neq 0$,
is clearly
\begin{align*}
\textsc{m}(x,t)&=\frac{f(tx)}{f(x)}\\
&=\frac{c\log(tx)}{c\log{x}}\\
&= 1+\frac{\log{t}}{\log{x}}\\
&= 1+\frac{f(t)}{f(x)}.
\end{align*}

The additive homogeneity function reads
\begin{align*}
\textsc{a}(x,t)&=f(x+t)-f(x)\\
&={c \log(x+t)}-{c\log{x}}\\
&= c \log{\left(1+\frac{t}{x}\right)}\\
&=  f{\left(1+s\right)},
\end{align*}
where $s:=\frac{t}{x}$.
Here, similarly as for the exponential homogeneity function
in case of power functions,
the additive homogeneity function
$\textsc{a}$ is homogeneous of degree $0$.
 
The logarithmic homogeneity function reads
\begin{align*}
\textsc{l}(x,t)&=f(tx)-f(x)\\
&={c \log(tx)}-{c\log{x}}\\
&=c \log{\left(\frac{tx}{x}\right)}\\
&=c \log{t}\\
&= f(t),
\end{align*}
a logarithmic function, which depends only on the variable $t$. 

The exponential homogeneity function reads
\begin{align*}
\eh(x,t)&=\frac{f(x+t)}{f(x)}\\
&=\frac{c\log{(x+t)}}{c\log{x}}\\
&=\frac{\log{(x+t)}}{\log{x}},
\end{align*}
where $1 \neq x > 0$ and $x>-t$.
If $x=1$ and $t=0$, since formally $\eh(1,0)=\frac{0}{0}$,
we obtain by L'Hospital 
\begin{align*}
\eh(1,0)&=\left. \frac{\frac{1}{{x+t}}}{\frac{1}{{x}}}  \right|_{x=1}^{t=0}
\\
&= \left. \frac{x}{x+t}  \right|_{x=1}^{t=0}
&=1.
\end{align*}


\subsection{Sine function}
The homogeneity function of $f:\R \to (0, +\infty)$ defined by $f(x)=\sin{x}$
is clearly
\begin{align*}
\textsc{m}(x,t)&=\frac{f(tx)}{f(x)}\\
&=\frac{\sin(tx)}{\sin{x}}\\
&= 
\frac{e^{i(tx -\frac{\pi}{2})} + e^{-i(tx -\frac{\pi}{2})}}{e^{i(x -\frac{\pi}{2})}
+ e^{-i(x -\frac{\pi}{2})}}.
\end{align*}

The additive homogeneity function reads
\begin{align*}
\textsc{a}(x,t)&=f(x+t)-f(x)\\
&={\sin(x+t)}-{\sin{x}}\\
&=\sin{x} \cos{t}+\sin{t}\cos{x}-\sin{x}\\
&= \sin{x} (\cos{t}-1)+ \sin{t} \cos{x} \\
&= 2 \sin{\frac{t}{2}} \cos{\left(\frac{t}{2}+x\right)}.
\end{align*}

Use the sum-difference-formula
\begin{align*}
\sin{A}-\sin{B}
&=
2\cos{\left(\frac{A+B}{2}\right)}
\sin{\left(\frac{A-B}{2}\right)} 
\end{align*}
and $\cos(x-\frac{\pi}{2})=\sin(x)$ to verify the last step.


The logarithmic homogeneity function reads
\begin{align*}
\textsc{l}(x,t)&=f(tx)-f(x)\\
&=\sin(tx)-\sin{x}\\
&= -2 \sin\left(\frac{x}{2}(1-t)\right) \cos\left(\frac{x}{2}(1+t)\right).
\end{align*}

The exponential homogeneity function reads
\begin{align*}
\eh(x,t)&=\frac{f(x+t)}{f(x)}\\
&=\frac{\sin(x+t)}{\sin{x}}\\
&=\frac{\sin{x}\cos{t} + \sin{t} \cos{x}}{\sin{x}}\\
&= \cos{t}+ \sin{t} \cot{x}.
\end{align*}

\subsection{Cosine function}
The homogeneity function of $f:\R \to \R$ defined by $f(x)=\cos{x}$
is clearly
\begin{align*}
\textsc{m}(x,t)&=\frac{f(tx)}{f(x)}\\
&=\frac{\cos(tx)}{\cos{x}}\\
&=\frac{e^{i tx} + e^{-i tx}}{e^{ix}+ e^{-i x}}.
\end{align*}

For integer values of $t$, 
we can express $cos(tx)$ using Chebyshev polynomials of the first kind \cite{LiuDas2025}, denoted by $T_t$.
These polynomials are defined by $T_t (x) = \cos(t \arccos(x) )$ for $t \in \N_0$, 
and satisfy
\begin{align*}
\cos(tx) = T_t (x),
\end{align*}
where $t$ is the degree of the polynomial, 
Alternatively,  $T_t$ can be introduced
by the recurrence relation
\begin{equation*}
T_t(x) 
= 2x T_{t-1}(x)
-
 T_{t-2}(x) (x)
\end{equation*}
with the two base cases $T_0 (x)=1$ and $T_1 (x)=x$.

The additive homogeneity function reads
\begin{align*}
\textsc{a}(x,t)&=f(x+t)-f(x)\\
&={\cos(x+t)}-{\cos{x}}\\
&=\cos{x} \cos{t}-\sin{t}\sin{x}-\cos{x}\\
&= \cos{x} (\cos{t}-1)- \sin{x} \sin{t}\\
&= -2 \sin{\left(x+\frac{t}{2}\right)} \sin{\frac{t}{2}}.
\end{align*}

The logarithmic homogeneity function reads
\begin{align*}
\textsc{l}(x,t)&=f(tx)-f(x)\\
&=
\cos{(tx)}-\cos{x}\\
&=
-2 \sin\left(\frac{x (t-1)}{2}\right) 
\sin\left(\frac{x (t+1)}{2}\right).
\end{align*}

To obtain the final form of the additive and logarithmic, respectively, homogeneity function, use the sum-difference-formula
\begin{equation*}
\cos{A}-\cos{B}
=
-2\sin{\left(\frac{A+B}{2}\right)}
\sin{\left(\frac{A-B}{2}\right)}. 
\end{equation*}

The exponential homogeneity function reads
\begin{align*}
\eh(x,t)&=\frac{f(x+t)}{f(x)}\\
&=\frac{\cos(x+t)}{\cos{x}}\\
&=\frac{\cos{x}\cos{t} - \sin{x} \sin{t}}{\cos{x}}\\
&= \cos{t}- \tan{x} \sin{t} \\
&=\frac{e^{i t} + e^{- i (2x+t)}}{1 + e^{- 2i x }}.\\
\end{align*}

\section{Homogeneity and the Harmonic Mean}
The concept of homogeneity is not limited to single variable functions,
but works as well for functions of several variables.
To illustrate this, we analyze the homogeneity properties of the harmonic mean
$H: (0,+\infty)^2 \to (0,+\infty)$ given by
\begin{align*}
H(x,y) = \frac{2xy}{x+y}.
\end{align*}

\subsection{Multiplicative Homogeneity}
The harmonic mean is \emph{positively homogeneous of degree 1}, meaning its multiplicative homogeneity function is simply:
\begin{align*}
\textsc{m}_H(x,y,t) := \frac{H(tx,ty)}{H(x,y)} = t,
\end{align*}
which is the identity function on $(0, +\infty)$. 
This shows perfect scalability under uniform rescaling of the variables.

\subsection{Additive homogeneity}
The additive homogeneity function 
reveals how $H$ changes under translation

\begin{align*}
\textsc{a}_H(x,y,t) &:= H(x+t,y+t) - H(x,y) \\
&= \frac{t[(x-y)^2 + t(x+y)]}{(x+y+2t)(x+y)}.
\end{align*}
This shows that the translation effect depends on both the scale $t$ and the initial disparity between $x$ and $y$.

\subsection{Exponential Homogeneity}
The exponential homogeneity function, 
defined by
$\eh_H (x,y,t):=\frac{H(x+t, y+t)}{H(x,y)}$,
measures the relative change under translation,
reads here
\begin{align*}
	\eh_H (x,y,t)=\frac{(x+t)(y+t)(x+y)}{xy(x+y+2t)}.
\end{align*}

\subsection{Logarithmic Homogeneity}
For homogeneous means, the logarithmic homogeneity function simplifies to
\begin{align*}
\textsc{l}_H(x,y,t) &:= H(tx,ty) - H(x,y) \\
&= (t-1)H(x,y)\\
&= \frac{2(t-1)xy}{x+y}.
\end{align*}

By the relation between the harmonic, arithmetic and geometric mean, 
namely, for all $x,y >0$,
\begin{align*}
H(x,y) &= \frac{2 xy}{x+y}\\
&= \frac{G^2(x,y)}{A(x,y)},
\end{align*}
where
$H, G: (0, + \infty)^2 \to (0, + \infty)$ 
with $G(x,y):=\sqrt{xy}$ and $A:\R^2 \to \R$ with $A(x,y):=\frac{x+y}{2}$,
and the fact that the arithmetic mean is both translative and homogeneous, 
i.e.,
\begin{align*}
A(x+t, y+t) &= t+ A(x,y), \quad x,y,t \in \R,
\end{align*}
and 
\begin{align*}
A(tx, ty) &= t A(x,y), \quad x,y,t \in \R,
\end{align*}
the homogeneity functions of $H$ can be expressed in terms of the corresponding ones of $G^2$ and $A$.
Using this, we obtain
\begin{align*}
 \ah_{G^2}(x,y,t) = t(x+y)+t^2, \quad x,y,t >0,
\end{align*}
and thus
\begin{align*}
 \ah_{H} (x,y,t) = \frac{t(x^2+y^2) +t^2 (x+y)}{(x+y+2t) (x+y)}, \quad x,y,t >0.
\end{align*}

\begin{itemize}
\item \textbf{Scaling}: Multiplicative homogeneity shows $H$ scales linearly with its arguments
\item \textbf{Translation}: Additive homogeneity reveals the change depends on both the translation magnitude $t$ and initial disparity $(x-y)^2$
\item \textbf{Symmetry}: All forms respect the symmetry $H(x,y) = H(y,x)$
\item \textbf{Special Case}: When $x=y$, by reflexivity of means, all homogeneity functions simplify significantly:
\begin{align*}
\textsc{a}_H(x,x,t) &= t \\
\eh_H(x,x,t) &= 1 + \frac{t}{x} \\
\textsc{l}_H(x,x,t) &= (t-1)x
\end{align*}
\end{itemize}

\section{Homogeneity and Means}
Let $I \subset \R$ be a non-void interval. 
A function $M: I^2 \to I$ such that
\begin{align}\label{eq:internality}
\min(x,y) \leq M(x,y) \leq \max(x,y)
\end{align}
for all $x,y \in I$ is called a mean in $I$, and property \eqref{eq:internality}
is called \term{internality} (cf. \cite{Matkowski2003}, \cite{Matkowski2012}).
In this section we apply the idea of generalized homogeneity
to means.
As a prestep, recall that minimum and maximum can be expressed in terms of the arithmetic mean and the absolute value, namely for all $x,y \in I$
\begin{align*}
\min(x,y)&=\frac{x+y - |x-y|}{2}\\
\max(x,y)&= \frac{x+y + |x-y| }{2}.
\end{align*}
This helps to deal with the homogeneity functions of $\min$ and $\max$ more explicitly.
We obtain 
\begin{align*}
\ah_{\min} (x,y,t) &= \min (x+t, y+t) - \min (x,y) \\
                 &= t;\\
\eh_{\min} (x,y,t) &=  \frac{\min (x+t, y+t)}{\min (x,y)} \\
&=\frac{x+y+2t-|x-y|}{x+y-|x-y|};\\
\lh_{\min} (x,y,t) &=  {\min (tx, ty)} - {\min (x,y)} \\
&=  \frac{1}{2}\left( x(t-1) + y(t-1)- (|t|-1) |x-y|\right);\\  
\mh_{\min} (x,y,t) &=  \frac{\min (tx, ty)}{\min (x,y)}  \\
&=  \frac{t (x+y) - |t| |x-y| }{x+y-|x-y|} \\
&= 
\begin{cases}
t & t \geq 0,\\
t \frac{\max(x,y)}{\min(x,y)}, \qquad t<0,
\end{cases} 
\end{align*}
and
\begin{align*}
\ah_{\max} (x,y,t) &= \max (x+t, y+t) - \max (x,y) \\
                 &= t;\\
\eh_{\max} (x,y,t) &=  \frac{\max (x+t, y+t)}{\max (x,y)} \\
&=\frac{x+y+2t+|x-y|}{x+y+|x-y|};\\
\lh_{\max} (x,y,t) &=  {\max (tx, ty)} - {\max (x,y)} \\
&=  \frac{1}{2}\left( x(t-1) + y(t-1)+ (|t|-1) |x-y|\right);\\  
\mh_{\max} (x,y,t) &=  \frac{\max (tx, ty)}{\max (x,y)}  \\
&=  \frac{t (x+y) + |t| |x-y| }{x+y-|x-y|}   \\
&= 
\begin{cases}
t & t \geq 0,\\
t \frac{\min(x,y)}{\max(x,y)},\qquad  t<0
\end{cases}
\end{align*}

Let us assume $M$ to be a positive mean. 
Replacing $x$ by $tx$ and $y$ by $ty$ in \eqref{eq:internality}, we obtain
\begin{align*}
\min(tx,ty) \leq M(tx,ty) \leq \max(tx,ty),
\end{align*}
and thus, by the notion of multiplicative homogeneity functions,
\begin{align*}
\frac{\mh_{\min} (x,y,t)}{\min(x,y)} \leq\frac{\mh_{M} (x,y,t)}{M(x,y)}
\leq \frac{\mh_{\max} (x,y,t)}{\max(x,y)}.
\end{align*}

\section{Homogeneity functions of only one variable}

We wonder when a homogeneity function is independent of $x$ or $t$.
A differentiable function is independent of a variable if its partial derivative with respect to this variable vanishes.
We obtain the following

\begin{theorem}[Multiplicative Homogeneity]
Let \( I \subset (0, +\infty) \) be an interval and \( f:I \to (0, +\infty) \) a differentiable function. 
Define the \emph{multiplicative homogeneity function} of \( f \) as:
\[
\mh_f(x,t) = \frac{f(tx)}{f(x)}, \quad x \in I, t > 0.
\]
Then:
\begin{enumerate}
    \item \( \mh_f \) is independent of \( x \) \textbf{if, and only if,} \( f \) is a power law:
    \[
    f(x) = Cx^k \quad \text{for constants } C > 0, k \in \mathbb{R}.
    \]
    \item \( \mh_f \) is independent of \( t \) \textbf{if, and only if,} \( f \) is constant.
\end{enumerate}
\end{theorem}

\begin{proof}
\begin{enumerate}
    \item \textbf{(Independence of \( x \))}:
    (\(\Rightarrow\)) Suppose \( \mh_f \) is independent of \( x \). Then:
    \[
    \frac{\partial}{\partial x} \left( \frac{f(tx)}{f(x)} \right) = \frac{t f'(tx) f(x) - f(tx) f'(x)}{f(x)^2} = 0.
    \]
    This implies the key identity:
    \[
    t \frac{f'(tx)}{f(tx)} = \frac{f'(x)}{f(x)}. \quad (*)
    \]
    Let \( k := \frac{x f'(x)}{f(x)} \) (logarithmic derivative). Then:
    \[
    \text{From } (*): \quad t \cdot \frac{k}{tx} = \frac{k}{x} \implies \frac{k}{x} = \frac{k}{x},
    \]
    which holds for all \( k \). Thus, \( k \) must be constant. Integrating \( \frac{f'(x)}{f(x)} = \frac{k}{x} \) gives:
    \[
    \ln f(x) = k \ln x + \ln C \implies f(x) = Cx^k.
    \]
    
    (\(\Leftarrow\)) If \( f(x) = Cx^k \), then:
    \[
    \mh_f(x,t) = \frac{C(tx)^k}{Cx^k} = t^k,
    \]
    which is independent of \( x \).

    \item \textbf{(Independence of \( t \))}:
    (\(\Rightarrow\)) Assume \( \mh_f \) is independent of \( t \). Then:
    \[
    \frac{\partial}{\partial t} \left( \frac{f(tx)}{f(x)} \right) = \frac{x f'(tx)}{f(x)} = 0.
    \]
    Since \( f(x) > 0 \), we have \( x f'(tx) = 0 \). For \( x \neq 0 \), \( f^\prime(tx) = 0 \) for all \( t > 0 \), so \( f \) is constant.
    
    (\(\Leftarrow\)) If \( f \) is constant, \( \mh_f(x,t) = 1 \) is trivially independent of \( t \).
\end{enumerate}
\end{proof}


\begin{theorem}[Multiplicative Homogeneity (Non-Differentiable Case)]
Let \( I \subset (0, +\infty) \) be an interval and \( f:I \to (0, +\infty) \) a continuous function. 
Define the \emph{multiplicative homogeneity function} as:
\[
\mh_f(x,t) = \frac{f(tx)}{f(x)}, \quad x \in I, t > 0.
\]
Then:
\begin{enumerate}
    \item \( \mh_f \) is independent of \( x \) \textbf{if and only if} \( f \) is a power law:
    \[
    f(x) = Cx^k \quad \text{for constants } C > 0, k \in \mathbb{R}.
    \]
    \item \( \mh_f \) is independent of \( t \) \textbf{if and only if} \( f \) is constant.
\end{enumerate}
\end{theorem}

\begin{proof}
\begin{enumerate}
    \item \textbf{(Independence of \( x \))}:
    (\(\Rightarrow\)) If \( \mh_f(x,t) = g(t) \) for some \( g \), then:
    \[
    f(tx) = f(x) g(t).
    \]
    This is the well-known semi-pexiderized multiplicative Cauchy functional equation. For continuous \( f \), the only non-zero solutions are power functions:
    \[
    f(x) = Cx^k,
    \]
    where \( C > 0 \) and \( k \in \mathbb{R} \) are constants (proven by taking logarithms and reducing to the additive Cauchy equation).
    
    (\(\Leftarrow\)) If \( f(x) = Cx^k \), then:
    \[
    \mh_f(x,t) = \frac{C(tx)^k}{Cx^k} = t^k,
    \]
    which depends only on \( t \).

    \item \textbf{(Independence of \( t \))}:
    (\(\Rightarrow\)) If \( \mh_f(x,t) = h(x) \), then:
    \[
    f(tx) = f(x) h(x).
    \]
    For fixed \( x \), the right-hand side is independent of \( t \), so \( f(tx) \) must be constant in \( t \). Thus:
    \[
    f(tx) = f(x) \quad \forall t > 0,
    \]
    implying \( f \) is constant (set \( t = y/x \) for \( y \in I \)).
    
    (\(\Leftarrow\)) If \( f \) is constant, \( \mh_f(x,t) = 1 \) trivially.
\end{enumerate}
\end{proof}


For additive homogeneous functions we have
\begin{theorem}
Let $I \subset \R$ be an interval and $f:I \to (0, +\infty)$ be a differentiable function.
Then the additive homogeneity function of $f$ is independent of the variable
\begin{enumerate}
	\item 
	$x$, if and only if
	$f^\prime$ is $t$-periodic
    for all $t$ 
    (thus $f$ is  $t$-periodic plus a linear term); 
	\item 
	$t$, if, and only if, $f$ is constant on an interval.
\end{enumerate}
\end{theorem}

\begin{proof} 
\begin{enumerate}
    \item 
    Since $f$ is differentiable and $\ah_f$ independent of $x$, we have
    \begin{equation*}
\frac{\partial}{\partial x} \ah_f(x,t) = 
1 f^\prime (x+t) -  f^\prime (x)
=0,
\end{equation*}
whence $f^\prime (x+t) =  f^\prime (x)$ as claimed.

  \item 
Assume $\ah_f$ independent of $t$.
Since $f$ is differentiable we have
    \begin{equation*}
\frac{\partial}{\partial t} \ah_f(x,t) = 
 f^\prime (x+t) - 0
=0.
\end{equation*}
Thus, $f^\prime (x+t) =0$,
which means that $f$ is constant on an interval.
\end{enumerate} 
\end{proof}


\begin{theorem}[Exponential Homogeneity]
Let \( I \subset \mathbb{R} \) be an interval and \( f:I \to (0, +\infty) \) a differentiable function. 
Define the \emph{exponential homogeneity function} of \( f \) as:
\[
\eh_f(x,t) = \frac{f(x+t)}{f(x)}, \quad x, x+t \in I.
\]
Then:
\begin{enumerate}
    \item \( \eh_f \) is independent of \( x \) \textbf{if, and only if,}
    \[
    \eh_f(x,t) = e^{t \cdot \frac{f'(x)}{f(x)}},
    \]
    i.e., when \( f \) is an exponential function \( f(x) = Ce^{kx} \) with \( C,k \in \mathbb{R} \).

    \item \( \eh_f \) is independent of \( t \) \textbf{if, and only if,} \( f^\prime \equiv 0 \) on \( I \), 
    i.e., when \( f \) is constant.
\end{enumerate}
\end{theorem}

\begin{proof}
\begin{enumerate}
    \item \textbf{(Independence of \( x \))}:
    Assume \( \eh_f \) is independent of \( x \). Then \( \frac{\partial}{\partial x} \eh_f(x,t) = 0 \). 
    Computing the derivative:
    \[
    \frac{\partial}{\partial x} \left( \frac{f(x+t)}{f(x)} \right) = \frac{f^\prime(x+t)f(x) - f(x+t)f^\prime(x)}{f(x)^2} = 0.
    \]
    This implies:
    \[
    \frac{f^\prime(x+t)}{f(x+t)} = \frac{f^\prime(x)}{f(x)} \quad  x, t.
    \]
    Thus, the logarithmic derivative \( \frac{f^\prime}{f} \) is constant. Let \( k := \frac{f^\prime(x)}{f(x)} \).
    Integrating yields:
    \[
    \ln f(x) = kx + C \quad \Rightarrow \quad f(x) = Ce^{kx}.
    \]
    Conversely: For \( f(x) = Ce^{kx} \), we have:
    \[
    \eh_f(x,t) = \frac{e^{k(x+t)}}{e^{kx}} = e^{kt},
    \]
    which is clearly independent of \( x \).

    \item \textbf{(Independence of \( t \))}:
    Assume \( \eh_f \) is independent of \( t \). Then \( \frac{\partial}{\partial t} \eh_f(x,t) = 0 \).
    Computing the derivative:
    \[
    \frac{\partial}{\partial t} \left( \frac{f(x+t)}{f(x)} \right) = \frac{f'(x+t)}{f(x)} = 0.
    \]
    Since \( f(x) > 0 \), we get \( f^\prime(x+t) = 0 \) for all \( x+t \in I \), hence \( f^\prime \equiv 0 \) and \( f \) is constant.
    
    Conversely: If \( f \) is constant, then \( \eh_f(x,t) = 1 \) is trivially independent of \( t \).
\end{enumerate}
\end{proof}


\begin{theorem}[Logarithmic Homogeneity]
Let $I \subset \mathbb{R}$ be an interval and $f: I \to \mathbb{R}$ be a differentiable function. Define the \emph{logarithmic homogeneity function} as:
\[
\textsc{l}_f (x, t) = f(tx) - f(x), \quad x \in I, t > 0.
\]
Then:

\begin{enumerate}
    \item \textbf{Independence of $x$}:
    \begin{enumerate}
        \item If $0 \in I$, then $\textsc{l}_f$ is independent of $x$ \textbf{if, and only if,} $f$ is constant.
        \item If $0 \notin I$, then $\textsc{l}_f$ is independent of $x$ \textbf{if, and only if,}:
        \[
        f(x) = 
        \begin{cases}
        k\log x + C & \text{for } I \subset (0, +\infty) \\
        k\log(-x) + C & \text{for } I \subset (-\infty, 0)
        \end{cases}
        \]
        where $k, C \in \mathbb{R}$.
    \end{enumerate}
    
    \item \textbf{Independence of $t$}: 
    $\textsc{l}_f$ is independent of $t$ \textbf{if and only if} $f$ is constant.
\end{enumerate}
\end{theorem}

\begin{proof}
\begin{enumerate}
    \item \textbf{Independence of $x$}:
    
    ($\Rightarrow$) Assume $\textsc{l}_f$ is independent of $x$. Then:
    \begin{align*}
        f(tx) - f(x) &= g(t) \quad \text{for some function } g \\
        \frac{d}{dx}[f(tx) - f(x)] &= 0 \\
        tf'(tx) - f'(x) &= 0 \\
        tf'(tx) &= f'(x)
    \end{align*}
    
    Let $x = 1$:
    \[
    f'(t) = \frac{f'(1)}{t}
    \]
    Integrating gives:
    \[
    f(t) = f'(1)\log|t| + C.
    \]
    
    \begin{enumerate}
        \item For $0 \in I$, differentiability at $0$ requires $f^\prime (1) = 0$, so $f$ is constant.
        \item For $0 \notin I$, we obtain logarithmic forms as stated.
    \end{enumerate}
    
    ($\Leftarrow$) For the logarithmic forms:
    \[
    \textsc{l}_f(x,t) = k\log(tx) - k\log x = k\log t
    \]
    which is independent of $x$. The constant case is trivial.
    
    \item \textbf{Independence of $t$}:
    
    ($\Rightarrow$) Assume $\textsc{l}_f$ is independent of $t$:
    \[
    f(tx) - f(x) = h(x)
    \]
    Setting $t = 1$ implies $h(x) = 0$, so $f(tx) = f(x)$ for all $t > 0$. Thus $f$ is constant.
    
    ($\Leftarrow$) If $f$ is constant, $\textsc{l}_f(x,t) = 0$ is trivially independent of $t$.
\end{enumerate}
\end{proof}

Under natural assumptions on $T$,
the preceeding results can be improved by simpy reducing 
the problem to semi-pexiderized Cauchy equation.

\begin{remark}
    Let \(T \subset \mathbb{R}\) satisfy \(0 \in T\) and \(-I + I \subset T\) for some interval \(I\) around \(0\).  
    Let \(f : \mathbb{R} \to \mathbb{R}\) and consider the four homogeneity functions

    \begin{align*}
        \textsc{a}_f(x,t) &= f(x+t) - f(x), &
        \textsc{m}_f(x,t) &= \frac{f(tx)}{f(x)}, \\[4pt]
        \eh_f(x,t) &= \frac{f(x+t)}{f(x)}, &
        \textsc{l}_f(x,t) &= f(tx) - f(x).
    \end{align*}

    In cases involving division, assume \(f(x) \neq 0\) for all \(x\).

    \begin{enumerate}
        \item \textbf{Additive homogeneity:}  
              \begin{itemize}
                  \item \emph{Only depends on \(x\):} If \(\textsc{a}_f(x,t) = g(x)\) for all \(t \in T\), then setting \(t=0\) gives \(g(x) \equiv 0\). Hence \(f(x+t) = f(x)\) for all \(x,t \in T\). If \(T\) contains an interval around \(0\), then \(f\) is constant.
                  \item \emph{Only depends on \(t\):} If \(\textsc{a}_f(x,t) = a(t)\), then \(f(x+t) = f(x) + a(t)\). With \(x=0\): \(a(t) = f(t) - f(0)\). Then \(h(x):= f(x)-f(0)\) satisfies \(h(x+t) = h(x) + h(t)\) for all \(x,t \in T\). So \(h\) is additive on \(T\).
              \end{itemize}

        \item \textbf{Multiplicative homogeneity:}  
              \begin{itemize}
                  \item \emph{Only depends on \(x\):} If \(\textsc{m}_f(x,t) = g(x)\), set \(t=1\): \(g(x) \equiv 1\). Then \(f(tx) = f(x)\) for all \(t \in T\). Taking \(t=0\) (if \(0 \in T\)) gives \(f(0) = f(x)\), so \(f\) is constant. If \(0 \notin T\), from \(f(tx)=f(x)\) for all \(t \in T\) and \(x\), choosing \(x \neq 0\) and varying \(t\) over a nontrivial set can still force constancy.
                  \item \emph{Only depends on \(t\):} If \(\textsc{m}_f(x,t) = m(t)\), then \(f(tx) = m(t) f(x)\). For \(x=1\): \(f(t) = m(t) f(1)\). Hence \(m(t) = f(t)/c\) with \(c = f(1) \neq 0\), and thus \(c \cdot f(tx) = f(t) f(x)\). Then \(g(x):= f(x)/c\) satisfies \(g(tx) = g(t) g(x)\), the multiplicative Cauchy equation.
              \end{itemize}

        \item \textbf{Exponential homogeneity:}  
              \begin{itemize}
                  \item \emph{Only depends on \(x\):} If \(\eh_f(x,t) = g(x)\), set \(t=0\): \(g(x) \equiv 1\). Then \(f(x+t) = f(x)\) for all \(x,t\), so \(f\) is constant.
                  \item \emph{Only depends on \(t\):} If \(\eh_f(x,t) = e(t)\), then \(f(x+t) = e(t) f(x)\). Set \(x=0\): \(e(t) = f(t)/a\) with \(a = f(0) \neq 0\). Hence \(a \cdot f(x+t) = f(x) f(t)\). Then \(g(x):= f(x)/a\) satisfies \(g(x+t) = g(x) g(t)\), the exponential Cauchy equation.
              \end{itemize}

        \item \textbf{Logarithmic homogeneity:}  
              \begin{itemize}
                  \item \emph{Only depends on \(x\):} If \(\textsc{l}_f(x,t) = g(x)\), set \(t=1\): \(g(x) \equiv 0\). Hence \(f(tx) = f(x)\) for all \(t \in T\), \(x\). Taking \(t=0\) (if \(0 \in T\)) yields \(f(0) = f(x)\), so \(f\) constant. If \(0 \notin T\), analogous argument as in the multiplicative case applies.
                  \item \emph{Only depends on \(t\):} If \(\textsc{l}_f(x,t) = \ell(t)\), then \(f(tx) = f(x) + \ell(t)\). Set \(x=1\): \(\ell(t) = f(t) - c\) with \(c = f(1)\). Hence \(h(x):= f(x)-c\) satisfies \(h(tx) = h(x) + h(t)\), the logarithmic Cauchy equation.
              \end{itemize}
    \end{enumerate}

    In summary, dependence on \(x\) alone typically forces \(g\) to be trivial (0 or 1) and ultimately constancy of \(f\), whereas dependence on \(t\) alone leads to one of the classical Cauchy-type functional equations (additive, multiplicative, exponential, or logarithmic) for a suitably normalized version of \(f\).
\end{remark}


\section{Equality problem of homogeneity functions}
If functions $f$ and $g$ coincide, they obviously have the same homogeneity function. 
On the other hand, 
if homogeneity functions of same type coincide, their generator
do not necessarily have to agree.

\begin{remark}[Equality of Multiplicative Homogeneity Functions]
Let $I, T \subset \R$ be non-empty intervals such that $T\cdot I \subset I$.
Assume 
\begin{eqnarray*}
{\textsc{m}_f}(x,t) = {\textsc{m}_g} (x,t), \qquad x \in I, t \in T,
\end{eqnarray*}
for given functions $f,g: I \to (0, +\infty)$.
Then $f$ and $g$ are proportional, i.e. there is $c>0$ such that
\begin{eqnarray*}
f (x) = c g (x), \qquad x \in I.
\end{eqnarray*}
\end{remark}

\begin{proof}
Assuming the equality of two multiplicative homogeneity functions for non-vanishing functions $f$ and $g$ we have
\begin{eqnarray*}
\frac{f(tx)}{f(x)} = \frac{g(tx)}{g(x)}, \qquad x \in I, t \in T,
\end{eqnarray*}
 whence
\begin{eqnarray*}
\frac{f(tx)}{g(tx)} = \frac{f(x)}{g(x)}, \qquad x \in I, t \in T.
\end{eqnarray*}
Thus, the function $h:=\frac{f}{g}$ is homogeneous of degree zero, which means that $h$ is constant, say $h(x)=c$.
By the definition of $h$, we have $f= c g$ as claimed.
\end{proof}

\begin{remark}[Equality of Additive Homogeneity Functions]
Let $I, T$ be real non-empty intervals such $I+T \subset I$.
Assume 
\begin{eqnarray*}
\textsc{a}_f (x,t) = \textsc{a}_g (x,t), \qquad x \in I, t \in T,
\end{eqnarray*}
for given functions $f,g: I \to \R$.
Then $f$ and $g$ agree up to a constant, i.e. there is $c \in \R$ such that
\begin{eqnarray*}
f (x) = c + g (x), \qquad x \in I.
\end{eqnarray*}
\end{remark}

\begin{proof}
Assuming the equality of two additive homogeneity functions, we have
\begin{eqnarray*}
f(x+t)-f(x) = g(x+t)-g(x), \qquad x \in I, t \in T,
\end{eqnarray*}
 whence
\begin{eqnarray*}
f(x+t)-g(x+t) =f(x) -g(x),, \qquad x \in I, t \in T.
\end{eqnarray*}
Thus, the function $h:={f}-{g}$ is $t$-periodic for all 
$t \in T$, which means, since $T$ is an interval, that $h$ is constant, 
say $h(x)=c$.
By the definition of $h$, we have $f= c+ g$ as claimed.
\end{proof}

\begin{remark}[Equality of Logarithmic Homogeneity Functions]
Let $I, T$ be real non-empty intervals such that $T \cdot I \subset I$.
Assume 
\begin{eqnarray*}
\textsc{l}_f (x,t) = \textsc{l}_g (x,t), \qquad x \in I, t \in T,
\end{eqnarray*}
for given functions $f,g: I \to \R$.
Then $f$ and $g$ agree up to a constant, i.e. there is $c \in \R$ such that
\begin{eqnarray*}
f (x) = c + g (x), \qquad x \in I.
\end{eqnarray*}
\end{remark}

\begin{proof}
Assuming the equality of two logarithmic homogeneity functions, we have
\begin{eqnarray*}
f(tx)-f(x) = g(tx)-g(x), \qquad x \in I, t \in T,
\end{eqnarray*}
 whence
\begin{eqnarray*}
f(tx)-g(tx) =f(x) -g(x),, \qquad x \in I, t \in T.
\end{eqnarray*}
Thus, the function $h:={f}-{g}$ is $0$-homogeneous,
which means that $h$ is constant, 
say $h(x)=c$.
By the definition of $h$, we have $f= c+ g$ as claimed.
\end{proof}

\begin{remark}[Equality of Exponential Homogeneity Functions]
Let $I, T$ be real non-empty intervals such that $I + T \subset I$..
Assume 
\begin{eqnarray*}
{\eh_f}(x,t) = \eh_g (x,t), \qquad x \in I, t \in T,
\end{eqnarray*}
for given functions 
$f,g: I \to (0, +\infty)$.
Then $f$ and $g$ are proportional, i.e. there is positive $c$ such that
\begin{eqnarray*}
f (x) = c  g (x), \qquad x \in I.
\end{eqnarray*}
\end{remark}

\begin{proof}
Assuming the equality of two exponential homogeneity functions for non-vanishing functions $f$ and $g$ we have
\begin{eqnarray*}
\frac{f(x+t)}{f(x)} 
= \frac{g(x+t)}{g(x)}, \qquad x \in I, t \in T,
\end{eqnarray*}
 whence
\begin{eqnarray*}
\frac{f(x+t)}{g(x+t)} 
= \frac{f(x)}{g(x)}, \qquad x \in I, t \in T.
\end{eqnarray*}
Thus, the function $h:=\frac{f}{g}$ is $t$-periodic for all $t \in T$,
which means that $h$ is constant, say $h(x)=c$.
By the definition of $h$, 
we have $f= c g$ as claimed.
\end{proof}

\section{The Cauchy equations and Homogeneity Functions}
The four types of Cauchy equations
can be expressed in terms of homogeneity functions -
how exactly we see in the following

\begin{remark}
Let $I \subset \R$ be a non-empty interval and $f: I \to \R$ a function.
Then $f$ is additive if, and only if,
\begin{enumerate}
    \item 
its additive homogeneity function satisfies
\begin{align*}
\textsc{a}_f (x,t) = f(t), \qquad x,t;   
\end{align*}

\item 
its exponential homogeneity function satisfies
\begin{align*}
\eh_f (x,t) = 1+ \frac{f(t)}{f(x)}, \qquad x,t;   
\end{align*}
\end{enumerate}
\end{remark}

\begin{proof}
    \begin{enumerate}
    \item
    \textbf{$(\Rightarrow)$} 
    Assume $f$ is additive, i.e.
\begin{align*}
f(x+y)= f(x) + f(y)  
\end{align*}
for all $x,y$.
 Thus, for all $x,t$,
\begin{align*}
\textsc{a}_f (x,t) &= f(x+t)-f(x)\\
&=
f(x)+f(t)-f(x) \\
&=
f(t),
\end{align*}
 for all $x,t$ as claimed. \\
\textbf{$(\Leftarrow)$} Vice versa, assuming $\textsc{a}_f (x,t)=f(t)$, we have, by the definition of the additive homogeneity function,
\begin{align*}
f(x+t)-f(x)
&=
f(t),
\end{align*}
which means that $f$ is additive.
\item
 \textbf{$(\Rightarrow)$} 
    Assume $f$ is additive.
Thus, for all $x,t$,
\begin{align*}
\eh_f (x,t) &= \frac{f(x+t)}{f(x)}\\
&=
\frac{f(x)+ f(t)}{f(x)} \\
&=
1+ \frac{f(t)}{f(x)}.
\end{align*}
  \textbf{$(\Leftarrow)$}   Vice versa, assuming $\eh_f (x,t)= 1+ \frac{f(t)}{f(x)}$,
     we have, by the definition of the exponential homogeneity function,
\begin{align*}
\frac{f(x+t)}{f(x)}
&=
1+ \frac{f(t)}{f(x)}
\end{align*}
for all $x,t$.
Multiplying here both sides by $f(x)$, the claim follows.
\end{enumerate}
\end{proof}
Similarly, the exponential Cauchy equation can be expressed in terms of
additive and exponential, respectively, homogeneity function.
\begin{remark}
Let $I \subset \R$ be a non-empty interval and $f: I \to \R$ a function.
Then $f$ is exponential if, and only if,
\begin{enumerate}
    \item 
its additive homogeneity function satisfies
\begin{align*}
\textsc{a}_f (x,t) =f(x) (f(t)-1), \qquad x,t;   
\end{align*}

\item 
its exponential homogeneity function satisfies
\begin{align*}
\eh_f (x,t) = f(t), \qquad x,t;   
\end{align*}
\end{enumerate}
\end{remark}

\begin{proof}
    \begin{enumerate}
    \item
   \textbf{$(\Rightarrow)$} 
    Assume $f$ is exponential, i.e.
\begin{align*}
f(x+y)= f(x)  f(y)  
\end{align*}
for all $x,y$.
 Thus, for all $x,t$,
\begin{align*}
\textsc{a}_f (x,t) &= f(x+t)-f(x)\\
&=
f(x)f(t)-f(x) \\
&=
f(x) (f(t)-1),
\end{align*}
as claimed. \\
\textbf{$(\Leftarrow)$} Vice versa, assuming $\textsc{a}_f (x,t)= f(x) (f(t)-1)$, we have, by the definition of the additive homogeneity function,
\begin{align*}
f(x+t)-f(x)
&=
f(x) (f(t)-1),
\end{align*}
for all $x,t$. 
Adding here $f(x)$ to both sides,
gives us that $f$ is exponential.
\item
 \textbf{$(\Rightarrow)$} 
    Assume $f$ is exponential.
Thus, for all $x,t$,
\begin{align*}
\eh_f (x,t) &= \frac{f(x+t)}{f(x)}\\
&=
\frac{f(x) f(t)}{f(x)} \\
&=
f(t).
\end{align*}
\textbf{$(\Leftarrow)$}  Vice versa, assuming $\eh_f (x,t)=f(t)$,
     we have, by the definition of the exponential homogeneity function,
\begin{align*}
\frac{f(x+t)}{f(x)}
&=
f(t)
\end{align*}
for all $x,t$.
Multiplying here both sides by $f(x)$ it follows that $f$ is exponential.
\end{enumerate}
\end{proof}
Similarly, the two remaining Cauchy equations can be expressed in terms of the multiplicative and logarithmic homogeneity functions, respectively.

\begin{remark}
Let $I \subset \R$ be a non-empty interval and
$f: I \to (0, + \infty)$ a function.
Then $f$ is logarithmic if, and only if,
\begin{enumerate}
    \item 
its multiplicative homogeneity function satisfies
\begin{align*}
\textsc{m}_f (x,t) =1+ \frac{f(t)}{f(x)}, \qquad x,t;   
\end{align*}

\item 
for its logarithmic homogeneity function holds
\begin{align*}
\textsc{l}_f (x,t) = f(t), \qquad x,t;   
\end{align*}
\end{enumerate}
\end{remark}

\begin{proof}
    \begin{enumerate}
    \item
  \textbf{$(\Rightarrow)$} 
    Assume $f$ is logarithmic, i.e.
\begin{align*}
f(x y)= f(x) + f(y)  
\end{align*}
for all $x,y$.
 Thus, for all $x,t$,
\begin{align*}
\textsc{m}_f (x,t) &= \frac{f(tx)}{f(x)}\\
&=
\frac{f(t) + f(x)}{f(x)}\\
&=
1 + \frac{f(t)}{f(x)},
\end{align*}
as claimed. \\
\textbf{$(\Leftarrow)$} Vice versa, assuming $\textsc{m}_f (x,t)= 1 + \frac{f(t)}{f(x)}$,
we have, by the definition of the multiplicative homogeneity function,
\begin{align*}
\frac{f(tx)}{f(x)}
&=
1 + \frac{f(t)}{f(x)},
\end{align*}
for all $x,t$. 
Multiplying both sides by $f(x)$,
gives us that $f$ is logarithmic.
\item
 \textbf{$(\Rightarrow)$} 
    Assume $f$ is logarithmic.
Thus, for all $x,t$,
\begin{align*}
\textsc{l}_f (x,t) &= {f(tx)}-{f(x)}\\
&=
{f(t)+f(x)}-{f(x)} \\
&=
f(t).
\end{align*}
\textbf{$(\Leftarrow)$}    Vice versa, assuming $\textsc{l}_f (x,t)=f(t)$,
     we have, by the definition of the logarithmic homogeneity function,
\begin{align*}
 {f(tx)}-{f(x)}
&=
f(t)
\end{align*}
for all $x,t$.
Adding here $f(x)$ to both sides, gives us that $f$ is logarithmic.
\end{enumerate}
\end{proof}

\begin{remark}
Let $I \subset \R$ be a non-empty interval and
$f: I \to (0, + \infty)$ a function.
Then $f$ is multiplicative if, and only if,
\begin{enumerate}
    \item 
its multiplicative homogeneity function satisfies
\begin{align*}
\textsc{m}_f (x,t) =f(t), \qquad x,t;   
\end{align*}

\item 
its logarithmic homogeneity function satisfies
\begin{align*}
\textsc{l}_f (x,t) = f(x)(f(t)-1), \qquad x,t;   
\end{align*}
\end{enumerate}
\end{remark}

\begin{proof}
    \begin{enumerate}
    \item
    \textbf{$(\Rightarrow)$} 
    Assume $f$ is multiplicative, i.e.
\begin{align*}
f(xy)= f(x)  f(y)  
\end{align*}
for all $x,y$.
 Thus, for all $x,t$,
\begin{align*}
\textsc{m}_f (x,t) &= \frac{f(tx)}{f(x)}\\
&=
\frac{f(t) f(x)}{f(x)} \\
&=
f(t),
\end{align*}
 as claimed. \\
 \textbf{$(\Leftarrow)$} Vice versa, assuming $\textsc{m}_f (x,t) = f(t)$, we have, by the definition of the multiplicative homogeneity function,
\begin{align*}
\frac{f(tx)}{f(x)}
&=
f(t)
\end{align*}
for all $x,t$. 
Multiplying here both sides by $f(x)$,
gives us that $f$ is multiplicative.
\item
 \textbf{$(\Rightarrow)$} 
    Assume $f$ is multiplicative.
Thus, for all $x,t$,
\begin{align*}
\textsc{l}_f (x,t) &= {f(tx)}-{f(x)}\\
&=
{f(t)  f(x)}-{f(x)} \\
&=
f(x)(f(t)-1).
\end{align*}
  \textbf{$(\Leftarrow)$}    Vice versa, assuming $\textsc{l}_f (x,t)=f(x)(f(t)-1)$,
     we have, by the definition of the logarithmic homogeneity function,
\begin{align*}
{f(tx)}-{f(x)}
&=
f(x)(f(t)-1)
\end{align*}
for all $x,t$.
Adding $f(x)$ to both sides gives that $f$ is multiplicative.
This concludes the proof.
\end{enumerate}
\end{proof}

\section{Homogeneity functions of Cauchy Quotients}
Here we consider the four types of Cauchy quotients as functionals 
and ask how the respective kind of homogeneity of the generator is transported.
Let us recall the Cauchy quotients: for an arbitrary function $f: (0, + \infty) \to (0, + \infty)$
and $k \in \N$, $k \geq 2$, arbitrarily fixed,
we define
\begin{align*}
A_f (x_1, \ldots, x_k)&=\frac{f(x_1)+ \cdots + f(x_k)}{f(x_1 + \cdots + x_k)}, \\
B_f (x_1, \ldots, x_k)&=\frac{f(x_1)\cdot \cdots \cdot f(x_k)}{f(x_1 + \cdots + x_k)}, \\
L_f (x_1, \ldots, x_k)&=\frac{f(x_1)+ \cdots + f(x_k)}{f(x_1 \cdot \cdots \cdot x_k)}, \\
P_f (x_1, \ldots, x_k)&=\frac{f(x_1) \cdot \cdots \cdot f(x_k)}{f(x_1 \cdot \cdots \cdot x_k)}, 
\end{align*}
as the additive, exponential (traditionally known as beta-type function), logarithmic and multiplicative Cauchy quotient.
Assume that $f$ has (multiplicative) homogeneity function $\mh_f=:m$, i.e. 
\begin{equation*}
f(tx)= m(x,t) f(x), \qquad x \in I,t \in T.
\end{equation*}
Then we have
\begin{equation*}
P_{f,k}(tx_1, \ldots, tx_k)= \frac{m(x_1, t) \cdots m(x_k, t)}{m(x_1 \cdots x_k, t^k)} P_{f,k} (x_1, \ldots, x_k), \qquad x_1, \ldots, x_k \in I, t \in T,
\end{equation*}
and
\begin{equation*}
B_{f,k}(tx_1, \ldots, tx_k)= \frac{m(x_1, t) \cdots m(x_k, t)}{m(x_1 +\cdots +x_k, t)} B_{f,k} (x_1, \ldots, x_k), \qquad x_1, \ldots, x_k \in I, t \in T.
\end{equation*}

More systematically:
\begin{align*}
&\text{Given the homogeneity functions:} \\
&\textsc{a}_f(x,t) = f(x+t) - f(x), \\
&\textsc{m}_f(x,t) = \frac{f(tx)}{f(x)}, \\
&\eh_f(x,t) = \frac{f(x+t)}{f(x)}, \\
&\textsc{l}_f(x,t) = f(tx) - f(x).
\end{align*}

\medskip

\begin{align*}
&\text{(1) Additive Cauchy quotient } A_f: \\
&A_f(x_1, \ldots, x_k) = \frac{f(x_1) + \cdots + f(x_k)}{f(x_1 + \cdots + x_k)}, \\
&A_f(x_1 + t, \ldots, x_k + t) = \frac{(f(x_1) + \textsc{a}_f(x_1,t)) + \cdots + (f(x_k) + \textsc{a}_f(x_k,t))}
       {f(x_1 + \cdots + x_k) + \textsc{a}_f(x_1 + \cdots + x_k, k t)}.
\end{align*}

\medskip

\begin{align*}
&\text{(2) Multiplicative Cauchy quotient } P_f: \\
&P_f(x_1, \ldots, x_k) = \frac{f(x_1) \cdots f(x_k)}{f(x_1 \cdots x_k)}, \\
&P_f(t x_1, \ldots, t x_k) = \frac{\textsc{m}_f(x_1, t) f(x_1) \cdots \textsc{m}_f(x_k, t) f(x_k)}{\textsc{m}_f(x_1 \cdots x_k, t^k) f(x_1 \cdots x_k)} \\
&= \frac{\textsc{m}_f(x_1,t) \cdots \textsc{m}_f(x_k,t)}{\textsc{m}_f(x_1 \cdots x_k, t^k)} P_f(x_1, \ldots, x_k).
\end{align*}

\medskip

\begin{align*}
&\text{(3) Exponential Cauchy quotient } B_f: \\
&B_f(x_1, \ldots, x_k) = \frac{f(x_1) \cdots f(x_k)}{f(x_1 + \cdots + x_k)}, \\
&B_f(x_1 + t, \ldots, x_k + t) &= \frac{\eh_f(x_1,t) f(x_1) \cdots \eh_f(x_k,t) f(x_k)}{\eh_f(x_1 + \cdots + x_k, k t) f(x_1 + \cdots + x_k)} \\
&\quad &= \frac{\eh_f(x_1,t) \cdots \eh_f(x_k,t)}{\eh_f(x_1 + \cdots + x_k, k t)} B_f(x_1, \ldots, x_k).
\end{align*}

\medskip

\begin{align*}
&\text{(4) Logarithmic Cauchy quotient } L_f: \\
&L_f(x_1, \ldots, x_k) = \frac{f(x_1) + \cdots + f(x_k)}{f(x_1 \cdots x_k)}, \\
&L_f(t x_1, \ldots, t x_k) = \frac{(f(x_1) + \textsc{l}_f(x_1,t)) + \cdots + (f(x_k) + \textsc{l}_f(x_k,t))}
       {f(x_1 \cdots x_k) + \textsc{l}_f(x_1 \cdots x_k, t^k)}.
\end{align*}
Note that especially the homogeneity properties of multiplicative and exponential Cauchy quotients are remarkable,
since they factor and respect the form of the same Cauchy quotient!

\section*{Decomposition into even and odd part}
A function $f: \R \to \R$ may be uniquely decomposed into  a sum of
an even and odd part, namely

\begin{equation}
f(x)=
\underbrace{\frac{f(x)+f(-x)}{2}}_{=:f_e(x)}+
\underbrace{\frac{f(x)-f(-x)}{2}}_{=:f_o(x)}, \qquad x \in \R.
\label{eq:even_odd}
\end{equation}
Obviously, $f_e$ is an even and $f_o$ an odd function, 
i.e. $f_e(-x)=f_e(x)$ and $f_o(-x)=-f_o(x)$ for all $x\in \R$.
By conjugation,
 from inside with logarithm and from outside with exponential, 
 this decomposition has some multiplicative pendant for positive functions.
Thus,  applying the exponential function to both sides of \eqref{eq:even_odd}
and noting that for every $x \in \R$ there is a unique $u \in (0, + \infty)$ such that $x=\log{u}$, we get
 
 \begin{equation}
e^{f(\log{u})}=
e^{\frac{f(\log{u})+f(-\log{u})}{2}}
e^{\frac{f(\log{u})-f(-\log{u})}{2}}, \qquad u \in (0, + \infty).
\end{equation}
 
 Since $-\log{u}=\log{\left(\frac{1}{u}\right)}$ for positive $u$, we have
 
   \begin{equation*}
e^{f(x)}=
\sqrt{e^{f(\log{u})\cot f(\log{\frac{1}{u}})}} \cdot
\sqrt{e^{f(\log{u})-f(\log{\frac{1}{u}})}},
\qquad u \in (0, + \infty),
\end{equation*}

whence

\begin{equation*}
e^{f(\log{u})}=
\sqrt{e^{f(\log{u}) \cdot f(\log{\frac{1}{u}})}}+
\sqrt{e^{f(\log{u})-f(\log{\frac{1}{u}})}},
\qquad u \in (0, + \infty).
\end{equation*}

Thus, the function $g:=\exp \circ f \circ \log: (0, + \infty) \to (0, + \infty)$ satisfies

\begin{equation*}
g(u)=
\underbrace{\sqrt{g(u)g\left( \frac{1}{u}\right)}}_{g_{me}(u)}
\underbrace{\sqrt{\frac{g(u)}{g\left( \frac{1}{u}\right)}}}_{g_{mo}(u)}
\qquad u \in (0, + \infty).
\end{equation*}

The function $g_{me}$ is called the multiplicative even part,
and $g_{mo}$ the multiplicative odd part of $g$.
A positive function which coincides with its multiplicative even part [with its multiplicatively odd part],
i.e. $g_{me}=g$ [$g_{mo}=g$], 
is called multiplicatively even [multiplicatively odd].

It follows that a multiplicatively even function is characterized by $g(u)=g(\frac{1}{u})$ (multiplicatively even functions remain invariant under inversion)
and a multiplicative odd function by $g(\frac{1}{u})=\frac{1}{g(u)}$.

In connection with the decomposition of a function into even and odd part the arithmetic mean plays a prominent role, namely

\begin{eqnarray*}
f_{e}(x)&=\frac{f(x)+f(-x)}{2}\\
&=A(f(x), f(-x))
\end{eqnarray*}
and
\begin{eqnarray*}
f_{o}(x)&=\frac{f(x)-f(-x)}{2}\\
&=A(f(x), -f(-x))
\end{eqnarray*}
where $A:\R^2 \to \R$ defined by $A(x,y)=\frac{x+y}{2}$ is the arithmetic mean.

Later on, 
when decomposing a positive function into multiplicative even and multiplicative odd part,
the geometric mean appeared naturally, since

\begin{eqnarray*}
g_{me}(u)&={\sqrt{g(u) g (1/u) }}\\
&=G(g(u), g\left(1/u\right) )
\end{eqnarray*}
and
\begin{eqnarray*}
g_{mo}(u)&={\sqrt{\frac{g(u)}{g\left( 1/u \right)}}}\\
&=G(g(u), 1/g(1/u))
\end{eqnarray*}
where $G:(0,+\infty)^2 \to (0,+\infty)$ defined by $G(x,y)=\sqrt{xy}$ is the geometric mean
(and reflection at the $y$-axis becomes inversion).

More generally, we may ask how a (possibly non quasi-arithmetic) mean $M: I^2 \to I$
gives rise to a suitable decomposition of a given function defined on this interval.

\section{Some examples:  decomposition of classical functions into odd and even part and their conjugates }
To understand a little better the decomposition of a function defined on the reals
into odd and even part and how conjugation gives rise to a positive function 
with multiplicative odd and multiplicative even part, we consider some classical functions.

\subsection{Sine function}
Let $f: \R \to [-1, 1]$ be $f(x)=\sin{x}$.
It is known that the sine function is odd, namely, it holds
$f(-x)=-f(x)$ for all $x \in \R$.
Its exponential conjugate $g:=\exp \circ f \circ \log $ reads $g: (0, + \infty) \to [\frac{1}{e},  e]$ with $g(u)=e^{\sin{\left( \log{u}\right)}}$
$g(\frac{1}{u})=\frac{1}{g(u)}$ for all $u \in (0, + \infty)$.

\subsection{Cosine function}
Let $g: \R \to [-1, 1]$ be $g(x)=\cos{x}$.
It is known that the cosine function is even, which means that it holds
$g(-x)=g(x)$ for all $x \in \R$.
Its exponential conjugate $h:=\exp \circ f \circ \log $ reads $h: (0, + \infty) \to [\frac{1}{e},  e]$ with $h(u)=e^{\cos{\left( \log{u}\right)}}$
satisfies $h\left(\frac{1}{u}\right)=h(u)$ for all $u \in (0, + \infty)$.

\subsection{Power functions}
Let $f: \R \to \R$ be $f(x)=x^n$ for $n \in \N$ fixed. 
Apparently, $f$ is even for $n$ even, and $f$ is odd for odd exponents. 
Thus,
$g: (0, + \infty) \to \R$ defined by
$g(u):=e^{f(\log{u})}=e^{(\log{u})^n}=u^{(\log{u})^{n-1}}$
is multiplicatively even when \( n \) is even, and multiplicatively odd when \( n \) is odd.

\subsection{Exponential functions}
Let $f: \R \to (0,+\infty)$ be defined by $f(x)=a^x$ for some positive $a$.
Its decomposition into even and odd part reads
$f_e (x)=\frac{a^x +a^{-x}}{2}$
and
$f_o (x)=\frac{a^x -a^{-x}}{2}$. 
Its exponential conjugate $g: (0,+\infty) \to (0,+\infty)$ defined by $g:=\exp \circ f \circ \log$ reads

\begin{eqnarray*}
g(u)&= e^{a^{\log{u}}}\\
&= u^{\log{a}}
\end{eqnarray*}
having multiplicatively even part
$g_{me}(u)=\sqrt{e^{u+\frac{1}{u}}}$,
and multiplicatively odd part
$g_{me}(u)=\sqrt{e^{u-\frac{1}{u}}}$.

\section{Means in a linear space and how to get their weight functions}
The classical notion of real means using inequalities can be generalized to mean in a linear space 
(cf. \cite{PalMer2018}).
In a nutshell, a mean in a linear space (of $k$ variables, $k \geq 2$) is an affine combination of its variables where coefficients may depend on the variables.
To illustrate this idea, we consider the case of $k=2$ variables. 
For simplicity, we suppress in our notation dependencies on the variables in the weight functions, for which we usually use Greek letters.
Thus, $\lambda$ denotes in fact a function of several variables, for instance $\lambda=\lambda (x,y)$ in the two-variable case.
A two-variable mean in a linear space $M: I^2 \to I$ is thus a function of the form
\begin{equation*}
M(x,y)= \lambda x + (1-\lambda)y
\end{equation*}
where $\lambda: I^2 \to [0, + \infty]$ is the weight function of the mean.
Since $M(x,y)=y + \lambda (x-y)$ for all $x,y \in I$, 
we easily obtain $\lambda=\frac{M(x,y)-y}{x-y}$ whenever $x \neq y$.
{On the diagonal, i.e., when the variables are all equal, the weight functions of a mean are indeterminate.}
How about the three variables case? 
Let $M: I^3 \to I$ be a three variable mean:
\begin{eqnarray*}
	M(x,y,z) &= \lambda x + \mu y + (1-(\lambda+\mu))z\\
	&=
	\lambda (x-z) + \mu (y-z) +z.
\end{eqnarray*} 
The problem to find their two weight functions $\lambda, \mu: I^2 \to [0, +\infty]$ is slightly more involved.
We have
\begin{align}
	M(x,x,z)= (\lambda+\mu) (x-z)  + z,\qquad { x \neq y},
	\label{eq:Mean3diag21}
\end{align}
thus
\begin{equation}
	\lambda+\mu=\frac{M(x,x,z)-z }{x-z}, \qquad x \neq z,
	\label{eq:Mean3diag21}
\end{equation}
and
\begin{eqnarray*}
	M(x,y,y) &= \lambda x + (1-\lambda) y\\
	&=
	\lambda (x-y) +y.
\end{eqnarray*}
Consequently, 
\begin{equation*}
	\lambda=\lambda(x,y,y)=\frac{M(x,y,y)-y}{x-y}, \qquad x \neq y,
\end{equation*}
and thus, by \eqref{eq:Mean3diag21},
\begin{equation*}
	\mu =
	\frac{M(x,x,z)-z}{x-z}-\frac{M(x,y,y)-y}{x-y}, \qquad x \neq y, x \neq z.
\end{equation*}

\begin{remark}
	For a mean of $k$ variables, $k \in \N, k \geq 2$, $M: I^k \to I$, $(x_1, \ldots, x_k) \mapsto M(x_1, \ldots, x_k)$ 
	the $i$-th weight function, which is the function in front of the $i$-th variable, 
	equals the mean evaluated at the $i$-th variable at the $i$-th position and the $k$-th at all other positions minus the $z$-th variable devided by the difference of the $i$-th and the $k$-th variable, for $1 \leq i \leq k-1$.
    After having all $k-1$ independent weight functions at our disposition, the remaining one in front of the last variable equals $1-(\lambda_1+ \cdots +\lambda_{k-1})$.	
   
\end{remark}

\begin{remark}
    Deriving the weight functions for a general three variable mean, 
    we did a serious mistake:
    putting $x=y$, we obtained an expression for $\lambda + \mu$,
    and putting $y=z$, an expression for $\lambda$.
    In the end (namely when we went back to our original equation), 
    these additional assumptions ("conditional diagonality") fell under the table.
    The problem here is that the 1st summand is defined for 
    $x \neq z$ and the 2nd summand for $x \neq y$
    Our derivations lack rigour.
    \end{remark}

\section{Means and their Weight Functions}
By the preceding section the weight function of a bivariable mean $M:I^2 \to I$ is given by
\begin{equation}
	\lambda (x,y) =
	\frac{M(x,y)-y}{x-y}, \qquad x \neq y;
	\label{eq:lambda}
\end{equation}
$\lambda (x,x)$ is undetermined.

It is of main interest to charaterize a mean by its weight function.
More generally, when $M$ is any bivariable function on some intervall $I$,
we call the function $\lambda$ defined by \eqref{eq:lambda} its weight function.

Let us consider next a toy example where $M$ is a beta-type function, i.e.
\begin{equation*}
	M(x,y) =
	\frac{f(x)f(y)}{f(x+y)}, \qquad x \neq y.
\end{equation*}
Its weight function is given by
\begin{equation*}
\lambda_f (x,y)=\frac{f(x) f(y)-y f(x+y)}{(x-y) f(x+y)}	
\end{equation*}

It is known 
that a bivariable beta-type function is a mean if $f(x)=2x$
resulting in the weight function
\begin{equation*}
	\lambda_H (x,y)=\frac{y}{x+y}.	
\end{equation*}
So a smooth additive function generates the harmonic mean and the corresponding weight $\lambda_H$ is a rational function.

If $f$ is exponential, the beta-type weight function becomes
\begin{equation*}
	\lambda_f (x,y)=\frac{1-y}{x-y}, \quad x \neq y,	
\end{equation*}
having the following relation with the weight function $\lambda_H$
\begin{equation*}
	\lambda_f (x,y)=\lambda_f (x+1-2y,1-y).	
\end{equation*}

\section{Summary}
Generalized homogeneity describes how a function transforms when its inputs are scaled or shifted. The four main types are additive, multiplicative, exponential, and logarithmic homogeneity. Each type characterizes a distinct way the function changes under such transformations.

These generalized homogeneity functions help unify different scaling behaviors and reveal fundamental properties of functions, including how homogeneity influences related functional constructions. This framework extends classical homogeneity concepts to capture a broader range of functional symmetries.

\end{document}